\documentclass[]{interact}
\usepackage[utf8]{inputenc}

\usepackage{epstopdf}
\usepackage{subfigure}%
\usepackage[caption=false]{subfig}
\usepackage{algorithm} 
\usepackage{algorithmic} 

\usepackage[numbers,sort&compress]{natbib}
\bibpunct[, ]{[}{]}{,}{n}{,}{,}
\makeatletter
\def\NAT@def@citea{\def\@citea{\NAT@separator}}
\makeatother

\theoremstyle{plain}
\newtheorem{theorem}{Theorem}[section]
\newtheorem{lemma}[theorem]{Lemma}

\theoremstyle{definition}
\newtheorem{definition}[theorem]{Definition}

\theoremstyle{remark}
\newtheorem{remark}{Remark}

\begin{document}


\title{Error estimation for the non-convex cosparse optimization problem}

\author{
\name{Zisheng Liu
\textsuperscript{a}\thanks{CONTACT Z. Liu Email: liuzisheng0710@163.com; Ting Zhang Email: zhangting314314@126.com},
Ting Zhang\textsuperscript{~b}}
\affil{\textsuperscript{a}School of Statistics and Big Data, Henan University of Economics and Law, Zhengzhou, China;\\
       \textsuperscript{b}School of Mathematics and Information Science, Henan University of Economics and Law, Zhengzhou, China.}
}

\maketitle

\begin{abstract}
When the signal does not have a sparse structure but has sparsity under a certain transformation domain,
Nam et al. \cite{NS} introduced the cosparse analysis model, which provides a dual perspective on the sparse representation model.
This paper mainly discusses the error estimation of non-convex $\ell_p(0<p<1)$ relaxation cosparse optimization model with noise condition.
Compared with the existing literature, under the same conditions, the value range of the $\Omega$-RIP constant $\delta_{7s}$ given in this paper is wider.
When $p=0.5$ and $\delta_{7s}=0.5$, the error constants $C_0$ and $C_1$ in this paper are better than those corresponding results in the literature \cite{Cand,LiSong1}.
Moreover, when $0<p<1$, the error results of the non-convex relaxation method are significantly smaller than those of the convex relaxation method.
The experimental results verify the correctness of the theoretical analysis and illustrate that the $\ell_p(0<p<1)$ method can provide robust reconstruction for cosparse optimization problems.
\end{abstract}


\begin{amscode}
90C25
\end{amscode}

\section{Introduction}
\label{sec:introduction}
\subsection{Sparse synthesis model}

Although signals may exist in a high-dimensional ambient space, they often possess a low-dimensional structure that can be exploited for efficient recovery.
Prior knowledge of the low-dimensional structure can be incorporated to aid in signal recovery.
Sparsity is a common form of prior information used in this context and is the basis for the emerging field of compressive sensing (CS \cite{Donoho1,Donoho2,CS1,CS2,CS3}).
Sparse recovery has found applications in various domains, including imaging, speech processing, radar signal processing, sampling, and others, for recovering signals with sparse representations \cite{c1, c2, c3, c4}.
The linear system under consideration is a typical example of a sparse recovery problem
\begin{eqnarray}\label{e1}
y=Ax,
\end{eqnarray}
where $y\in R^m$ is an observed vector, $A\in R^{m\times d}$ is a measurement matrix and $x\in R^d$ is an unknown signal which would be estimated from $y$. We say that a signal $x$ is $s$-sparse if $\|x\|_0 := \#\{i: x_i \neq 0\} \leq s$.

In the last nearly two decades, the application of CS has greatly improved the speed and efficiency of image reconstruction due to its ability to reconstruct images from severely undersampled signals.
For sparse synthetic model, if the vector $x$ is sufficiently sparse, under the assumption of the inconsistency of the measurement matrix $A$, $x$ can be accurately estimated by the problem \eqref{e2}
\begin{equation}\label{e2}
\begin{aligned}
\min_x&~ \|x\|_p\\
s.t.&~y = Ax.
\end{aligned}
\end{equation}
where $0 < p \leq 1$. The key idea is to use the $\ell_p (0 < p \leq 1)$ norm (for a vector $x \in R^d$, the $\ell_p (0 < p \leq 1)$ (quasi)-norm of a vector $x$ means $\|x\|_p = (\sum_i^d |x_i|^p)^{\frac{1}{p}}  )$ or quasi-norm to relax the $\ell_0$ norm (the number of nonzero entries).
This theory has been a very active field of recent researches with a wide range of applications,
including signal processing, medical imaging, binary measurement and other applications and statistics \cite{c19,c21,c22,LS,LT,IM}.


With the popularization and application of computer and information technology, especially the rapid expansion of the application scale of cloud computing, internet technology, communication technology, digital economy and other industries, a large amount of redundant non-sparse data will be generated at all times, such as sensor signals, image data, financial data, etc\cite{NS,GR,YM,RR}.
To analysis these these non-sparse data, Nam et al. \cite{NS} proposed a new model from the point of view of the duality of the sparse representation model, i.e. consider the number of zero elements to analyze the vector.

\subsection{Cosparse analysis model}
Compared to the sparse model, the cosparse analysis model \cite{c23,c5,c30,c62,Davies1,Song} is an alternative  approach.
This model can be viewed as a generalization of the traditional sparse analysis model, which represents a signal as a linear combination of a small number of atoms from a dictionary \cite{RR}.
From a mathematical point of view, this model can be described as follow.

For a possibly redundant analysis operator $\Omega\in R^{n\times d}(n\geq d)$, the analyzed vector $\Omega x$ is expected to be sparse.
It means that a signal $x\in R^d$ belongs to the cosparse analysis model with cosparsity $\ell=n-\|\Omega x\|_0$, where the quantity $\ell$ is the number of rows in $\Omega$ that are orthogonal to the signal, then $x$ is said to be $\ell$-cosparse.
Based on the above description, the specific definition of cosparse and cosupport \cite{c5} are defined as below.

\begin{definition}[Cosparse]\label{d1}
A signal $x \in R^d$ is said to be cosparse with respect to an analysis operator $\Omega\in R^{n\times d}$ if the analysis representation vector $\Omega x$ contains many zero elements. Further, the number of zero elements
\begin{eqnarray*}
\ell= n-\|\Omega x\|_0
\end{eqnarray*}
is called the cosparsity of $x$, we also say $x$ is $\ell$-cosparse.
\end{definition}

The corresponding cosupport is defined as below.
\begin{definition}[Cosupport]\label{d2}
For a given analysis operator $\Omega\in R^{n\times d}$
and its rows $\Omega_j\in R^d(1\leq j\leq n)$,
the cosupport is defined by
\begin{eqnarray*}
\Lambda:=\{j|\langle \Omega_j,x\rangle=0\}.
\end{eqnarray*}
\end{definition}
For simplicity here, we assume that $\Omega$ is a tight frame $\Omega^T \Omega =I$.
We remind the reader that a frame is defined as below.
\begin{definition}[\cite{c6,c7}]\label{d3}
Let $\Phi=\{\varphi_{i}\}_{i=1}^{N} \subset R^{n}$ be a vector sequence of the Hilbert space with $N\geq n$. If there exist constants $0<A\leq B<\infty$ such that
\begin{eqnarray}
A\|x\|^{2}\leq\sum_{i=1}^N|\langle x,\varphi_{i}\rangle|^2\leq B\|x\|^{2},~~~\forall x\in R^{n},
\end{eqnarray}
then $\Phi$ is referred to as a finite frame of $R^{n}$. The constants $A$ and $B$ in the above formula are known as the lower and upper bounds of the finite frame $\Phi$, respectively. They are considered to be the optimal bounds, with $A$ being the supremum in the lower bound and $B$ being the infimum in the upper bound. If $A=B$, then the frame $\Phi$ is called an $A$-tight frame. If $A=B=1$, then $\Phi$ is called a Parseval frame. If there exists a constant $C$ such that each meta-norm $\|\varphi_{i}\|=C$ of the frame $\Phi$, then $\Phi$ is called an iso-norm frame. In particular, for a tight frame, if $C=1$, it is referred to as a uniformly tight frame.
\end{definition}

The concept of cosparsity in the cosparse analysis model is based on the number of zero elements in the analysis representation vector $\Omega x$, rather than the number of non-zero elements. This is opposite to the viewpoint of the sparse synthesis model. When the cosparsity $\ell$ is large, meaning that the number of zeros in $\Omega x$ is close to $d$, we say that $x$ has a cosparse representation. The cosupport set is identified by removing rows from $\Omega$ for which $\langle \Omega_j,x\rangle\neq0$, until the index set $\Lambda$ remains unchanged and $|\Lambda|\geq\ell$.

Similar to the sparse model, if the analysis representation vector $\Omega x$ is sparse, in compressed sensing, one considers the $Ax + v = y$ ($v$ is the noise vector, and $\|v\|_2 \leq \sigma$), the recovery problem can be formulated as
\begin{equation}\label{e3}
\begin{aligned}
\min_x &~\|\Omega x\|_0\\
s.t.&~\|y - Ax\|_2\leq \sigma.
\end{aligned}
\end{equation}

Solving the minimization problem \eqref{e3} is NP-hard \cite{c5}, making approximation methods necessary.
Similar to the sparse model, one option is to use the greedy analysis pursuit (GAP) \cite{c23,c5,c62,RGiryes} approach,
which is inspired by the orthogonal matching pursuit (OMP) algorithm.
Another approach is to approximate the nonconvex $\ell_0$ norm (it is not a norm in the strict sense, but is only used to count the number of non-zero elements) with the convex $\ell_1$ norm, leading to the following relaxed problem known as analysis basis pursuit (ABP) \cite{ZhaoT}. In this case, $x$ can be estimated by solving a modified optimization problem
\begin{equation}\label{e8}
\begin{aligned}
&\min_x \|\Omega x\|_1 \\
&s.t.~\|y - Ax\|_2\leq\sigma,
\end{aligned}
\end{equation}
where $\|\cdot\|_1$ is the $\ell_1$ norm that sums the absolute values of a vector.
For this model, Cand\`{e}s et al. in \cite{Cand} give an error estimate of the cosparse analysis model under $\Omega$-RIP conditions, as shown below:
\begin{equation}\label{E25}
\|x - \hat{x}\|_2 \leq C_0\sigma + C_1\frac{\|\Omega_{S_0^c} x\|_1}{\sqrt{s}}.
\end{equation}
Using the approach in \cite{Cand}, for $0<\delta_{7s} \leq 0.5$,  one would get a estimation in \eqref{E25} with $C_0 = 62$ and $C_1 = 30$.
Compared with the result of \cite{Cand}, Li et al. \cite{LiSong1} give a better error result with $C_0 \approx 10.57$ and $C_1 \approx 5.06$.

\begin{remark}
For an integer $s$ with $1 \leq s \leq d$, the RIP \cite{CS1} and $\Omega$-RIP \cite{Cand} are defined as bellow,
where $\Omega$-RIP is a natural extension to the standard RIP.
\begin{definition}[\cite{CS1}]\label{d3}
For a measurement matrix $A\in R^{m\times d}$, the $s$-restricted isometry constant $\delta_s^A \in (0,1)$ of $A$ is the smallest quantity such that
\begin{eqnarray*}
 (1-\delta_s^A)\|x\|_2^2 \leq \|A x\|_2^2 \leq (1+\delta_s^A)\|x\|_2^2
\end{eqnarray*}
holds for all $s$-sparse signals $x$.
\end{definition}

\begin{definition}[\cite{Cand}]\label{d4}
Let $\Omega$ be a tight frame and $\sum_s$ be the set of all $s$-sparse vectors in $R^d$.
A measurement matrix $A$ is said to obey the restricted isometry property adapted to $\Omega$ (abbreviated $\Omega$-RIP) with the smallest constant $\delta_s \in (0,1)$ if
\begin{eqnarray*}
 (1-\delta_s)\|\Omega x\|_2^2 \leq \|A \Omega x\|_2^2 \leq (1+\delta_s)\|\Omega x\|_2^2
\end{eqnarray*}
holds for all $x\in \sum_s$.
\end{definition}
\end{remark}

\begin{remark}
It is easy to see that it's computationally difficult to verify the $\Omega$-RIP for a given deterministic matrix.
But for matrices with Gaussian, subgaussian, or Bernoulli entries, the $\Omega$-RIP condition will be satisfied with overwhelming probability provided that the numbers of measurements $m$ is on the order of $s \log(d/s)$.
In fact, for any $m \times d$ matrix $A$ obeying for any fixed $x\in R^d$
\begin{equation*}
  \mathbb{P} (|\|A x\|_2^2 - \|x\|_2^2|\geq \delta \|x\|_2^2)\leq c e^{-\gamma m \delta^2},~~\delta\in(0,1),
\end{equation*}
($c,\gamma$ are positive numerical constants) will satisfy the $\Omega$-RIP with overwhelming probability provided that $m \geq s \log(d/s)$ \cite{Cand}.
\end{remark}



\subsection{Non-convex methods for solving cosparse optimization problems}

Depending on the specific objective function being minimized, the $\ell_1$-norm regularization term, which promotes sparsity by penalizing the sum of the absolute values of the solution, leads to a convex optimization problem \eqref{e8}.
In contrast, the $\ell_0$-norm regularization term, which promotes sparsity by penalizing the number of non-zero values in the solution, leads to a non-convex optimization problem \eqref{e3}.
For the compressed sensing problem \eqref{e2}, the $\ell_p(0<p<1)$ minimization problem is one of the most popular non-convex relaxations.

Similarly, the non-convex relaxation of cosparse optimization can be formulated as below:
\begin{equation}\label{E1}
 \begin{split}
  &\min_x \|\Omega x\|_p^p\\
  &s.t.~\|Ax - y\|_2\leq\sigma,
   \end{split}
\end{equation}
where $0<p<1, \Omega \in R^{n\times d}, A\in R^{m\times d}, x\in R^d, y\in R^m$, $\sigma$ is a upper bound on the noise level $\|v\|_2$.
While determining the local minimizer of the problem is polynomially solvable in time, computing the global minimizer of the problem is NP-hard for any fixed $0<p<1$. Therefore, at least locally, solving \eqref{E1} is still significantly faster than solving \eqref{e3}.



In this paper, we first discuss the non-convex model of the cosparse optimization problem and give the error estimation of the model \eqref{E1} in the noisy environment.
Our error results are independent of algorithm design,
so the results of this paper are general and have a wide range of applications.
Second, under the same conditions, the range of $0<\delta_{7s}\leq 0.63$ given in this paper is wider than that of existing literature \cite{Cand} with $0<\delta_{7s}\leq 0.5$.
Even when $0<p < 0.9$, the value $\delta_{7s}$ given in this paper will not be lower than 0.71, and can reach 0.98.
Finally, when $p=0.5$ and $\delta_{7s}= 0.5$, the error constants $C_0=3.8273$ and $C_1=0.8840$ given in this paper are better than those
in the existing literature \cite{Cand}($C_0 = 62$, $C_1 = 30$) and \cite{LiSong1}($C_0 \approx 10.57$, $C_1 \approx 5.06$).
The experiment in this paper also shows that the non-convex error result is better than the error result in the convex case.

\subsection{Organization of the paper}

We consider the error bound of the non-convex cosparse optimization problem. The rest of the paper is organized as follows.
In Section 2 we review some basic theoretical knowledge.
In Section 3, we prove our main result with the help of some relevant lemmas and reasonable assumptions.
In Section 4, numerical experiments illustrate the validity of the error bounds given in this paper.
Section 5 is the conclusion.

\section{Preliminaries}

In view of the above problems, this paper studies the non-convex relaxation recovery theory of cosparse optimization under the noise condition $Ax + v = y $. In this paper, we consider  the following constrained $\ell_p(0<p<1)$ problem

For subsequent proofs, here are some inequality properties about the $\ell_p$ norm.

\begin{lemma}\label{Lem01}
Let $0< p \leq 1$, for any vector $x$ in $R^d$, one has
\begin{align*}
&\|x\|_2 \leq \|x\|_1 \leq \sqrt{d}\|x\|_2,\\
&\|x\|_p \leq d^{\frac{1}{p}-\frac{1}{2}}\|x\|_2.
\end{align*}
\end{lemma}
\begin{proof}
Firstly, $\|x\|_2^2 = \sum_{i = 1}^d |x_i|^2 \leq \left(\sum_{i = 1}^d |x_i| \right)^2 = \|x\|_1^2$, therefore, $\|x\|_2 \leq \|x\|_1$. Next, we prove$\|x\|_1 \leq \sqrt{d}\|x\|_2.$
By the Cauchy-Schwartz inequality, we have
\begin{equation*}\label{Cauchy}
\|x\|_1 = \sum_{i = 1}^d |x_i| \leq \left(\sum_{i = 1}^d 1 \right)^{\frac{1}{2}}\left(\sum_{i = 1}^d  |x_i|^2 \right)^{\frac{1}{2}} = \sqrt{d}\|x\|_2.
\end{equation*}


Secondly, since $0<p\leq1$, we have $0<\frac{p}{2}<1$ and $\frac{2}{2-p}>1$. Combined with the H\"{o}lder inequality, we can derive the following inequality
\begin{align*}
\|x\|_p^p &= \sum_{i=1}^d |x_i|^p \cdot1\\
                 &\leq \left(\sum_{i=1}^d  |x_i|^2\right)^{\frac{p}{2}} \left(\sum_{i=1}^d 1^{\frac{2}{2-p}}\right)^{1-\frac{p}{2}} \\
                &= d^{1- \frac{p}{2}}\|x\|_2^p.
\end{align*}
The proof is completed.
\end{proof}


\section{Related Lemma}
Let $h = \hat{x} - x $, where $x$ is the true solution of \eqref{e3} and $\hat{x}$ is the minimizer of \eqref{E1}.
For $S \subset\{1,2,\ldots, n\}$, denote by $\Omega_S$ the matrix $\Omega$ restricted to the rows indexed by $S$. $S^c$ to mean the complement of $S$.
We assume that the first $s$ coordinates of $\Omega h$ are the largest in absolute value and $|t_{s+1}|\geq |t_{s+2}|\geq \cdots\geq |t_n|$, where $t_i, i=1,2,\ldots,n$ represents the complement elements of $\Omega h$.
Let $S_0$ denote the set of indicators for the first $s$ largest components, then $|S_0|=s$.
The $S_0^c$ represents the complement set of the set $S_0$.
Let the set $S_0^c$  divided into $S_1, \ldots, S_j,\ldots,S_J$ with $|S_j|=M>s$, where $j,J,M$ are positive integers,
then we have
\begin{equation*}
  \Omega h = \Omega_{S_0} h +\Omega_{S_1} h+ \cdots + \Omega_{S_J} h = (t_1,\ldots,t_s,t_{s+1},t_{s+2},t_{s+M},\ldots,t_n).
\end{equation*}
For the convenience of subsequent proofs, we will use the following symbols:
\begin{equation}\label{E2-0}
\rho=s^{\frac{1}{8}}M^{\frac{1-p}{p}},\quad \eta=\frac{2}{\sqrt{s}} \|\Omega_{S_0^c} x\|_1,\quad \kappa=\frac{s}{M},\quad \gamma=\rho M^{\frac{1}{2} -\frac{1}{p}},
\end{equation}
where $0<p<1$.

Our goal is to bound the norm of $h$. We will derive this bound through a sequence of short lemmas. The first lemma follows directly from the fact that $\hat{x}$ is the minimizer.


\begin{lemma}\label{Lem1}
Let $0 < p\leq 1$. We have the following inequality:
 \begin{equation}\label{E2}
  \|\Omega_{S_0^c} h\|_p^p \leq 2\|\Omega_{S_0^c} x\|_p^p  + \|\Omega_{S_0} h\|_p^p .
\end{equation}
\end{lemma}

\begin{proof}
For the case of $p=1$, the result was obtained in Reference \cite{Cand}. Thus, we only need to consider the case of $0< p <1$. Since both $x$ and $\hat{x}$ are feasible points of (\ref{E1}), and $\hat{x}$ is the minimizer of of (\ref{E1}),
we must have $\|\Omega \hat{x}\|_p \leq \|\Omega x\|_p$. By the triangular inequality, we have that
\begin{equation*}
\begin{aligned}
\|\Omega_{S_0} x\|_p^p  + \|\Omega_{S_0^c} x\|_p^p  = \|\Omega x\|_p^p  &\geq \|\Omega \hat{x}\|_p^p \\
& =  \|\Omega x + \Omega h\|_p^p  \\
& = \|\Omega_{S_0} x + \Omega_{S_0} h\|_p^p  + \|\Omega_{S_0^c} x + \Omega_{S_0^c} h\|_p^p   \\
&\geq \|\Omega_{S_0} x \|_p^p  - \|\Omega_{S_0} h\|_p^p  - \|\Omega_{S_0^c} x \|_p^p + \|\Omega_{S_0^c} h\|_p^p .
\end{aligned}
\end{equation*}
The conclusion can be obtained after sorting out the above formula.
\end{proof}
For the proof of the main theorem in the paper, we cite the lemma in the literature \cite{Rick} and give a concise proof process.
\begin{lemma}\label{Lem1-1}
Let $0 < p \leq 1$. We have the following inequality:
\begin{equation}\label{E2-1}
\|\Omega_{S_0^c} h \|_p^p \leq \|\Omega_{S_0} h \|_p^p.
\end{equation}
\end{lemma}

\begin{proof}
By the triangle inequality for $\|\cdot\|_p^p$, we have
\begin{equation*}\label{E2-2}
\begin{aligned}
\left|\|\Omega x\|_p^p -  \|-\Omega_{S_0} h\|_p^p \right | \leq \|\Omega x + \Omega_{S_0} h\|_p^p
\end{aligned}
\end{equation*}
Since $\hat{x}$ is the minimizer of \eqref{E1} and $S_0 \cap S_0^c = \emptyset$, we have
\begin{equation*}
\begin{aligned}
& \|\Omega x\|_p^p -  \|\Omega_{S_0} h\|_p^p + \|\Omega_{S_0^c} h \|_p^p \\
\leq & \|\Omega x + \Omega_{S_0} h + \Omega_{S_0^c} h\|_p^p  \\
=     & \|\Omega x   + \Omega h\|_p^p\\
\leq & \|\Omega x\|_p^p.
\end{aligned}
\end{equation*}
The proof is completed.
\end{proof}

\begin{lemma}\label{Lem2}
Let $0<p<1$ and $\rho, \eta$ be defined as in \eqref{E2-0}. Then, we have
\begin{equation}\label{E3}
\sum_{j\geq1}\|\Omega_{S_{j+1}} h\|_p^2 \leq \rho^2 (\|\Omega_{S_0} h\|_2 + \eta)^2 \leq \rho^2 (\|h\|_2 + \eta)^2.
\end{equation}
\end{lemma}

\begin{proof}
Let the absolute value of each elements of the $\Omega_{S_{j+1}}h$ is denoted by $|(\Omega_{S_{j+1}} h)_k|$ ($1\leq k \leq M$ is a positive integer),
then the $|(\Omega_{S_{j+1}} h)_k|$ does not exceed the mean value of the component elements in the set $S _j$.
Therefore, we have the following inequality
\begin{equation}\label{E3-1}
|(\Omega_{S_{j+1}} h)_k| \leq \frac{\|\Omega_{S_{j}} h\|_1}{M}.
\end{equation}
Since
\begin{equation*}
\sum_{k=1}^M |(\Omega_{S_{j+1}} h)_k|^p \leq \frac{1}{M^p}\sum_{k=1}^M \|\Omega_{S_{j}} h\|_1^p,
\end{equation*}
we have
\begin{equation*}
\|\Omega_{S_{j+1}} h\|_p^p \leq M^{1-p} \|\Omega_{S_{j}} h\|_1^p,
\end{equation*}
which is
\begin{equation}\label{J1}
\|\Omega_{S_{j+1}} h\|_p^2 \leq M^{\frac{2-2p}{p}} \|\Omega_{S_{j}} h\|_1^2.
\end{equation}
According to Lemma \ref{Lem1}, and the inequality $\|x\|_2 \leq \|x\|_1 \leq \sqrt{d}\|x\|_2$, we have
\begin{equation*}
\begin{aligned}
\sum_{j\geq 1}\|\Omega_{S_{j+1}} h\|_p^2 & \leq M^{\frac{2-2p}{p}} \sum_{j\geq 1}\|\Omega_{S_{j}} h\|_1^2 \\
&= M^{\frac{2-2p}{p}} \|\Omega_{S_0^c} h\|_1^2\\
&\leq  M^{\frac{2-2p}{p}} (2 \|\Omega_{S_0^c} x\|_1  + \|\Omega_{S_0} h\|_1)^2\\
&\leq  M^{\frac{2-2p}{p}} (2 \|\Omega_{S_0^c} x\|_1  + \sqrt{s}\|\Omega_{S_0} h\|_2)^2\\
&=  s^{\frac{1}{4}}M^{\frac{2-2p}{p}}(\frac{2}{\sqrt{s}} \|\Omega_{S_0^c} x\|_1  + \|\Omega_{S_0} h\|_2)^2\\
&=  \rho^2(\eta  + \|\Omega_{S_0} h\|_2)^2,
\end{aligned}
\end{equation*}
where the second inequality follows from $\|x\|_1 \leq \sqrt{d}\|x\|_2$ for $x\in \mathbb{R}^d$ in Lemma \ref{Lem1}, and the last equality follows from the definitions of $\rho$ and $\eta$ in \eqref{E2-0}. Thus, we obtain that the first inequality in (\ref{E3}).
\end{proof}

\begin{lemma}\label{Lem4}
Let $0<p<1$. We have the following inequality:
\begin{equation}\label{E4-0}
\|\Omega_{S_{j+1}} h\|_2^2 \leq M^{1-\frac{2}{p}}\|\Omega_{S_{j}} h\|_p^2.
\end{equation}
\end{lemma}

\begin{proof}
For each $|(\Omega_{S_{j}} h)_{k}|$, we have
\begin{equation}\label{E4-1}
|(\Omega_{S_{j+1}} h)_k|^p \leq \frac{1}{M}\sum_{k=1}^M |(\Omega_{S_{j}} h)_{k}|^p = \frac{1}{M}\|\Omega_{S_{j}} h\|_p^p.
\end{equation}
Taking both sides of the inequality to the power of $\frac{2}{p}$ at the same time yields the following inequality
\begin{equation}\label{E4-2}
|(\Omega_{S_{j+1}} h)_k|^2 \leq \frac{1}{M^{\frac{2}{p}}}\|\Omega_{S_{j}} h\|_p^2.
\end{equation}
Summing both sides of the inequality at the same time gives the following result
\begin{equation}\label{E4-3}
\sum_{k=1}^M|(\Omega_{S_{j+1}} h)_k|^2 \leq \frac{M}{M^{\frac{2}{p}}}\|\Omega_{S_{j}} h\|_p^2.
\end{equation}
\eqref{E4-3} is equivalent to the following inequality
\begin{equation*}\label{E4-4}
\|\Omega_{S_{j+1}} h\|_2^2 \leq M^{1-\frac{2}{p}}\|\Omega_{S_{j}} h\|_p^2.
\end{equation*}
The proof is completed.
\end{proof}

\begin{lemma}\label{Lem7}
Let $0<p<1$ and $\rho, \eta, \kappa,\gamma$ be defined as in \eqref{E2-0}. Denote $S_{01} = S_0 \cup S_1$. If $A$ satisfies the $\Omega$-RIP, then we have
\begin{equation}\label{E5}
\|\Omega_{S_{01}}^T\Omega_{S_{01}} h \|_2^2 \leq \frac{4\sigma^2 + \kappa^{\frac{2 - p}{p}} (1+\delta_{M}) \| h\|_2^2 + \gamma^2 (1+\delta_{M})(\|h\|_2 + \eta)^2}{(1-\delta_{s+M})}.
\end{equation}
\end{lemma}

\begin{proof}
Since both $x$ and $\hat{x}$ are feasible points of (\ref{E1}), it follows that
\begin{equation}
\|A h\|_2 = \|(Ax - y) - (A\hat{x} - y)\|_2\leq \|Ax - y\|_2 + \|A\hat{x} - y\|_2 \leq 2\sigma.
\end{equation}
Thus, we have
\begin{equation}\label{E6}
\begin{aligned}
4\sigma^2 &\geq \|A h\|_2^2 \\
&= \|A\Omega^T \Omega h\|_2^2\\
&\geq  \|A\Omega_{S_{01}}^T \Omega_{S_{01}} h\|_2^2 - \sum _{j\geq 1}\|A\Omega_{S_{j+1}}^T \Omega_{S_{j+1}} h\|_2^2\\
&\geq (1-\delta_{s+M})\|\Omega_{S_{01}}^T\Omega_{S_{01}} h \|_2^2 - (1+\delta_{M})\sum _{j\geq 1}\|\Omega_{S_{j+1}}^T \Omega_{S_{j+1}} h\|_2^2\\
&\geq (1-\delta_{s+M})\|\Omega_{S_{01}}^T\Omega_{S_{01}} h \|_2^2 - M^{1-\frac{2}{p}}(1+\delta_{M})\sum _{j\geq 1}\|\Omega_{S_{j}} h\|_p^2\\
&\geq (1-\delta_{s+M})\|\Omega_{S_{01}}^T\Omega_{S_{01}} h \|_2^2 - M^{1-\frac{2}{p}}(1+\delta_{M})\left(\|\Omega_{S_1} h\|_p^2 +  \sum _{j\geq 1}\|\Omega_{S_{j+1}} h\|_p^2\right) ,\\
\end{aligned}
\end{equation}
where the third inequality follows from Definition \ref{d4}, and the fourth inequality follows from Lemma \ref{Lem4} and $\|\Omega_{S_{j+1}}^T \Omega_{S_{j+1}} h\|_2^2\leq \|\Omega_{S_{j+1}} h\|_2^2 $.

Moreover, based on \eqref{E2-1} in Lemma \ref{Lem1-1}, we have
\begin{equation}\label{J2}
\|\Omega_{S_1} h\|_p \leq \|\Omega_{S_0^c} h \|_p \leq \|\Omega_{S_0} h \|_p.
\end{equation}
By Lemma \ref{Lem01}, we have
\begin{equation}\label{J3}
\|\Omega_{S_{0}} h\|_p^2 \leq s^{\frac{2}{p}-1} \|\Omega_{S_{0}} h\|_2^2.
\end{equation}
Combining inequality \eqref{J2} with inequality \eqref{J3}, and using $\|\Omega_{S_0} h\|_2 \leq \|h\|_2$, we have
\begin{equation}\label{J4}
\|\Omega_{S_{1}} h\|_p^2 \leq  s^{\frac{2}{p}-1} \| h\|_2^2.
\end{equation}
Using the inequality \eqref{J4} and Lemma \ref{Lem2}, let $\kappa = \frac{s}{M}$, then, the last inequality of \eqref{E6} is
\begin{equation}\label{E6-1}
\begin{aligned}
&M^{1-\frac{2}{p}}(1+\delta_{M})\left(\|\Omega_{S_1} h\|_p^2 +  \sum _{j\geq 1}\|\Omega_{S_{j+1}} h\|_p^2\right)\\
\leq& \kappa^{\frac{2 - p}{p}} (1+\delta_{M}) \| h\|_2^2 + M^{1-\frac{2}{p}}(1+\delta_{M}) \sum _{j\geq 1}\|\Omega_{S_{j+1}} h\|_p^2\\
\leq& \kappa^{\frac{2 - p}{p}} (1+\delta_{M}) \| h\|_2^2 + \rho^2 M^{1-\frac{2}{p}}(1+\delta_{M})(\|h\|_2 + \eta)^2.~~~(\text{Lemma \ref{Lem2}})
\end{aligned}
\end{equation}
Finally, using $\gamma = \rho M^{\frac{1}{2} -\frac{1}{p}}$ and combining \eqref{E6-1}, the inequality \eqref{E6} is equivalent to
\begin{equation}\label{E6-2}
\begin{aligned}
4\sigma^2 \geq &  (1-\delta_{s+M})\|\Omega_{S_{01}}^T\Omega_{S_{01}} h \|_2^2  - \kappa^{\frac{2 - p}{p}} (1+\delta_{M}) \| h\|_2^2 - \gamma^2 (1+\delta_{M})(\|h\|_2 + \eta)^2.
\end{aligned}
\end{equation}
Proofs are completed by sorting out the inequality \eqref{E6-2}.
\end{proof}

\begin{lemma}\label{Lem5}
Let $0<p<1$ and $\rho$, $\eta$, $\gamma$ be defined as in \eqref{E2-0}. Then, the error vector $h$ satisfies
\begin{equation}\label{E7}
(1 - \kappa^{\frac{2 - p}{p}})\| h \|_2^2 \leq \|h\|_2 \|\Omega_{S_{01}}^T\Omega_{S_{01}} h\|_2 + \gamma^2 (\|h\|_2 + \eta)^2.
\end{equation}
\end{lemma}

\begin{proof}
By Lemma \ref{Lem4}, we have $\|\Omega_{S_{j+1}} h\|_2^2 \leq M^{1-\frac{2}{p}}\|\Omega_{S_{j}} h\|_p^2$, since $\Omega$ is an isometry, we have
\begin{equation}\label{E7-1}
\begin{aligned}
\| h \|_2^2 = \| \Omega h \|_2^2  &=  \|\Omega_{S_{01}} h\|_2^2 +  \sum _{j\geq 1}\|\Omega_{S_{j+1}} h\|_2^2 \\
& \leq \langle h, \Omega_{S_{01}}^T\Omega_{S_{01}} h\rangle + M^{1-\frac{2}{p}}\sum _{j\geq 1}\|\Omega_{S_{j}} h\|_p^2\\
&\leq \| h \|_2 \|\Omega_{S_{01}}^T\Omega_{S_{01}} h\|_2 + M^{1-\frac{2}{p}} (\|\Omega_{S_1} h\|_p^2 + \sum _{j\geq 1}\|\Omega_{S_{j+1}} h\|_p^2 ).\\
\end{aligned}
\end{equation}
Using $\kappa = \frac{s}{M}$,  $\|\Omega_{S_{1}} h\|_p^2 \leq  s^{\frac{2}{p}-1} \| h\|_2^2$ and Lemma \ref{Lem2}, then, the last inequality of \eqref{E7-1} is
\begin{equation}\label{E7-2}
\begin{aligned}
&M^{1-\frac{2}{p}}\left(\|\Omega_{S_1} h\|_p^2 +  \sum _{j\geq 1}\|\Omega_{S_{j+1}} h\|_p^2\right)\\
\leq& \kappa^{\frac{2 - p}{p}}\| h\|_2^2 + M^{1-\frac{2}{p}}\sum _{j\geq 1}\|\Omega_{S_{j+1}} h\|_p^2\\
\leq& \kappa^{\frac{2 - p}{p}} \| h\|_2^2 + \rho^2 M^{1-\frac{2}{p}}(\|h\|_2 + \eta)^2.~~~(\text{Lemma \ref{Lem2}})
\end{aligned}
\end{equation}
Using $\gamma = \rho M^{\frac{1}{2} -\frac{1}{p}}$, combining \eqref{E7-1} and \eqref{E7-2}, the proof is completed.
\end{proof}

\begin{lemma}\label{Lem8}
Let $0<p<1$ and $\eta$, $\gamma$ be defined as in \eqref{E2-0}. If $\gamma<\sqrt{\frac{1}{4} - \frac{1}{2}\kappa^{\frac{2-p}{p}}}$,
the lower bound of $\|\Omega_{S_{01}}^T\Omega_{S_{01}} h\|_2$ satisfies
\begin{equation}\label{E12}
\sqrt{1 - 4\gamma^2 - 2\kappa^{\frac{2 - p}{p}}} \| h \|_2  - 2\gamma \eta \leq \|\Omega_{S_{01}}^T\Omega_{S_{01}} h\|_2.
\end{equation}
\end{lemma}

\begin{proof}
For any nonnegative numbers $u,v$, we have the inequality $uv \leq \frac{u^2}{2} + \frac{v^2}{2}$. By $\|\Omega_{S_0} h \|_2 \leq \|h\|_2$ and Lemma \ref{Lem5}, we have
\begin{equation}\label{E10}
\begin{aligned}
(1 - \kappa^{\frac{2 - p}{p}})\| h \|_2^2 &\leq \|h\|_2 \|\Omega_{S_{01}}^T\Omega_{S_{01}} h\|_2 + \gamma^2 (\|h\|_2 + \eta)^2\\
&\leq \frac{\|h\|_2^2}{2} + \frac{\|\Omega_{S_{01}}^T\Omega_{S_{01}} h\|_2^2}{2} + \gamma^2(\|h\|_2 + \eta)^2 \\
&= \frac{\|h\|_2^2}{2} + \frac{\|\Omega_{S_{01}}^T\Omega_{S_{01}} h\|_2^2}{2} + \gamma^2 \|h\|_2^2 + 2\gamma^2 \|h\|_2 \eta + \gamma^2 \eta^2 \\
&\leq \frac{\|h\|_2^2}{2} + \frac{\|\Omega_{S_{01}}^T\Omega_{S_{01}} h\|_2^2}{2} + \gamma^2 \|h\|_2^2 + 2\gamma^2 \Big(\frac{\|h\|_2^2}{2} + \frac{\eta^2}{2}\Big) + \gamma^2 \eta^2.\\
\end{aligned}
\end{equation}
Simplifying, this yields
\begin{equation}\label{E11}
(1 - 4\gamma^2 - 2\kappa^{\frac{2 - p}{p}} ) \| h \|_2^2  \leq \|\Omega_{S_{01}}^T\Omega_{S_{01}} h\|_2^2 + 4\gamma^2\eta^2.
\end{equation}
From the equation \eqref{E11}, using $\gamma<\sqrt{\frac{1}{4} - \frac{1}{2}\kappa^{\frac{2-p}{p}}}$ and $\sqrt{u^2 + v^2} \leq u + v$ for any nonnegative numbers, we can further simply to get our desired lower bound.
\end{proof}

\begin{theorem}
Let $0<p<1$, $\alpha = 1 - 4\gamma^2 - 2\kappa^{\frac{2 - p}{p}}$, $\beta = \kappa^{\frac{1}{p} - \frac{1}{2}}$ and $\rho$, $\eta$, $\gamma$ be defined as in \eqref{E2-0}, if $A$ satisfies the $\Omega$-RIP condition and
\begin{equation}\label{E19-1}
\gamma<\sqrt{\frac{1}{4} - \frac{1}{2}\kappa^{\frac{2-p}{p}}},~~ \delta_{s+M} < 1 - \frac{(\beta + \gamma)^2}{\alpha}(1 + \delta_M),
\end{equation}
the solution $\hat{x}$ to \eqref{E1} obeys
\begin{equation}\label{E19}
\|x - \hat{x}\|_2 \leq \frac{2\sigma}{K_1} + \frac{K_2\eta}{K_1},
\end{equation}
where the constants $K_1$ and $K_2$ are depend on $\kappa$ and $\delta$.
\end{theorem}

\begin{proof}
By the inequality \eqref{E5} in Lemma \ref{Lem7} we have
\begin{equation}\label{E13}
\|\Omega_{S_{01}}^T\Omega_{S_{01}} h \|_2 \leq \frac{2\sigma + \kappa^{\frac{1}{p} - \frac{1}{2}} \sqrt{1+\delta_{M}} \| h\|_2 + \gamma \sqrt{1+\delta_{M}}(\|h\|_2 + \eta)}{\sqrt{1-\delta_{s+M}}}.\\
\end{equation}
Combining the equation \eqref{E12} with equation \eqref{E13}, we have
\begin{equation}\label{E14}
\begin{aligned}
&\sqrt{1 - 4\gamma^2 - 2\kappa^{\frac{2 - p}{p}}} \| h \|_2  - 2\gamma \eta \\
\leq &\frac{2\sigma + \kappa^{\frac{1}{p} - \frac{1}{2}} \sqrt{1+\delta_{M}} \| h\|_2 + \gamma \sqrt{1+\delta_{M}}(\|h\|_2 + \eta)}{\sqrt{1-\delta_{s+M}}},
\end{aligned}
\end{equation}
the above inequality easily implies
\begin{equation}\label{E15}
\begin{aligned}
&\left[\sqrt{(1 - 4\gamma^2 - 2\kappa^{\frac{2 - p}{p}})(1-\delta_{s+M})}  - (\kappa^{\frac{1}{p} - \frac{1}{2}} +\gamma)\sqrt{1 + \delta_{M}}\right]\|h\|_2  \\
\leq& 2\sigma  + \gamma\eta (2\sqrt{1-\delta_{s+M}} +  \sqrt{1 + \delta_{M}}),
\end{aligned}
\end{equation}
that is
\begin{equation}\label{E16}
\|h\|_2  \leq \frac{2\sigma  + \gamma\eta (2\sqrt{1-\delta_{s+M}} +  \sqrt{1 + \delta_{M}})}{\sqrt{(1 - 4\gamma^2 - 2\kappa^{\frac{2 - p}{p}})(1-\delta_{s+M})}  - (\kappa^{\frac{1}{p} - \frac{1}{2}} + \gamma)\sqrt{1 + \delta_{M}}}.
\end{equation}
Let
\begin{equation}\label{E17}
K_1 = \sqrt{(1 - 4\gamma^2 - 2\kappa^{\frac{2 - p}{p}})(1-\delta_{s+M})}  - (\kappa^{\frac{1}{p} - \frac{1}{2}} + \gamma)\sqrt{1 + \delta_{M}}
\end{equation}
and
\begin{equation}\label{E18}
K_2 = \gamma (2\sqrt{1-\delta_{s+M}} +  \sqrt{1 + \delta_{M}}),
\end{equation}
Using \eqref{E19-1}, we have $K_1>0$, then, the proof is completed.
\end{proof}


\section{Numerical experiments}
Since $\eta = \frac{2}{\sqrt{s}}\|\Omega_{S_0^c} x\|_1$, we have
\begin{equation}\label{E20}
\|x - \hat{x}\|_2 \leq \frac{2\sigma}{K_1} + \frac{2K_2}{K_1\sqrt{s}}\|\Omega_{S_0^c} x\|_1.
\end{equation}
Let $s = 100, M = 6s,~0<p<1$, $\delta_{7s} = 0.5,~\delta_{6s}= 0.4$, $\sigma = 10^{-4}$.
As shown in equation \eqref{E17} and \eqref{E18},  $K_1, K_2$ are two constants, and the $\ell_1$ norm of the nonprincipal component is proportional to $\eta$ according to the definition of $\eta$.
Then the inequality \eqref{E20} means:
the signal recovery error can be divided into two parts, one related to the noise energy $\sigma$ and the other related to the $\ell_1$ norm of the non-major component.
Where the term related to noise is a linear relationship and the term related to non-major components is a quadratic relationship.

For example, let $C_0 = \frac{2}{K_1}$ and $C_1 = \frac{2K_2}{K_1}$, inequality \eqref{E20} can be rewritten as
\begin{equation}\label{E21}
\|x - \hat{x}\|_2 \leq C_0\sigma + C_1\frac{\|\Omega_{S_0^c} x\|_1}{\sqrt{s}}.
\end{equation}
If we set $p = 0.5$, we have that whenever $\delta_{7s} \leq 0.5$ that the non-convex relaxation of cosparse optimization model \eqref{E1} reconstructs $\hat{x}$ satisfying
\begin{equation}\label{E22}
\|x - \hat{x}\|_2 \leq 3.8273\sigma + 0.8840\frac{\|\Omega_{S_0^c} x\|_1}{\sqrt{s}}.
\end{equation}
For the convex relaxation of cosparse optimization model \eqref{e8},
Cand\`{e}s et al. \cite{Cand} showed that if $A$ satisfies $\Omega$-RIP with $\delta_{7s} \leq 0.5$,  one would get a estimation in \eqref{E21}
with $C_0 = 62$ and $C_1 = 30$, that is the solution $\hat{x}$ to \eqref{e8} satisfies
\begin{equation}\label{E23}
\|x - \hat{x}\|_2 \leq 60\sigma + 30\frac{\|\Omega_{S_0^c} x\|_1}{\sqrt{s}}.
\end{equation}
And by Theorem 3.4 of \cite{LiSong1}, for the convex relaxation of cosparse optimization model \eqref{e8},
one would get a estimation with $C_0 \approx10.57$ and $C_1 \approx 5.06$, respectively
\begin{equation}\label{E24}
\|x - \hat{x}\|_2 \leq 10.57\sigma + 5.06\frac{\|\Omega_{S_0^c} x\|_1}{\sqrt{s}}.
\end{equation}
Through comparison, we found that in the error conclusion of this paper,
our results are significantly better than the established conclusions under the same conditions.

Table 1 gives results showing a larger range of $\delta_{7s}$ under the relaxation model \eqref{E1}.
When $p = 1$, the value of $\delta_{7s}$  can reach at least 0.63, which is better than the results of literature \cite{Cand}.
When $p = 0.7$, the range of $\delta_{7s}$  can reach 0.92, which is significantly better than the case of $p=1$.
And as $p$ decreases, the value range of $\delta_{7s}$ will continue to increase, for example, when $p = 0.1$, $\delta_{7s} = 0.98$.
More intuitively, when $p$ is fixed, it can be seen from Figure 1 that as $\delta_{7s}$ increases, $K_1$  decreases,
indicating that the theoretical results are consistent with the experimental results
\begin{table}[htbp]
\centering
\caption{We take a point every 0.01 steps of $\delta_{7s}$.}
\begin{tabular}{c|cccccc}
\hline
$p$                       & 0.1        & 0.3         & 0.5        & 0.7        & 0.9          & 1  \\
\hline
$\delta_{7s}$      & 0.98      & 0.98       & 0.97      &0.92       & 0.77        &0.63\\

$K_1$                  & 2.6E-2  & 1.9E-2   & 2.9E-2  & 6.1E-3  & 7.6E-3     & 3.5E-3  \\
\hline
\end{tabular}
\end{table}

Further, as can be seen from the Figure 2, when $p$ is fixed to be 0.1, 0.3, 0.5, 0.7 and  0.9,
the error changes with the change of $\eta$, and the error will not exceed 0.15, which shows that our error range is very stable.
And from the Figure 2 we also find a phenomenon that the error decreases as $p$ decreases.
In particular, when $0<p<1$, it can be seen from the figure that the error is smaller than $p=1$, which also shows that the non-convex error result in this paper is better than the error result in the convex case, that is for $p = 1$ of the cosparse optimization \eqref{E1}.

\begin{figure*}[h]
\centering
\begin{minipage}[t]{0.48\textwidth}
\centering
\includegraphics[width=\textwidth]{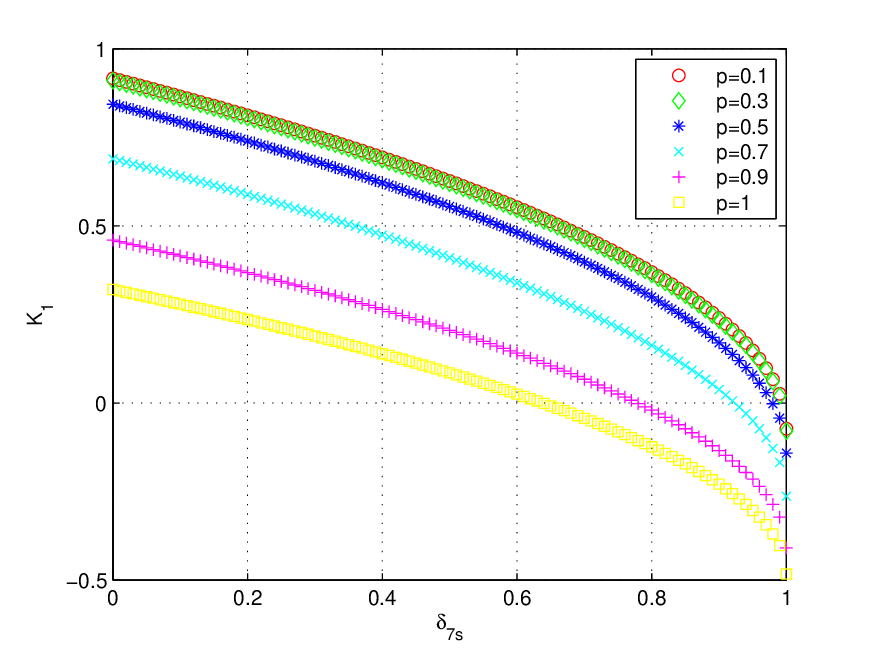}
\caption{$\delta_{7s}$ v.s. $K_1$ with $p$=0.1,0.3,0.5,0.7,0.9 and 1, respectively.}
\end{minipage}
\begin{minipage}[t]{0.48\textwidth}
\centering
\includegraphics[width=\textwidth]{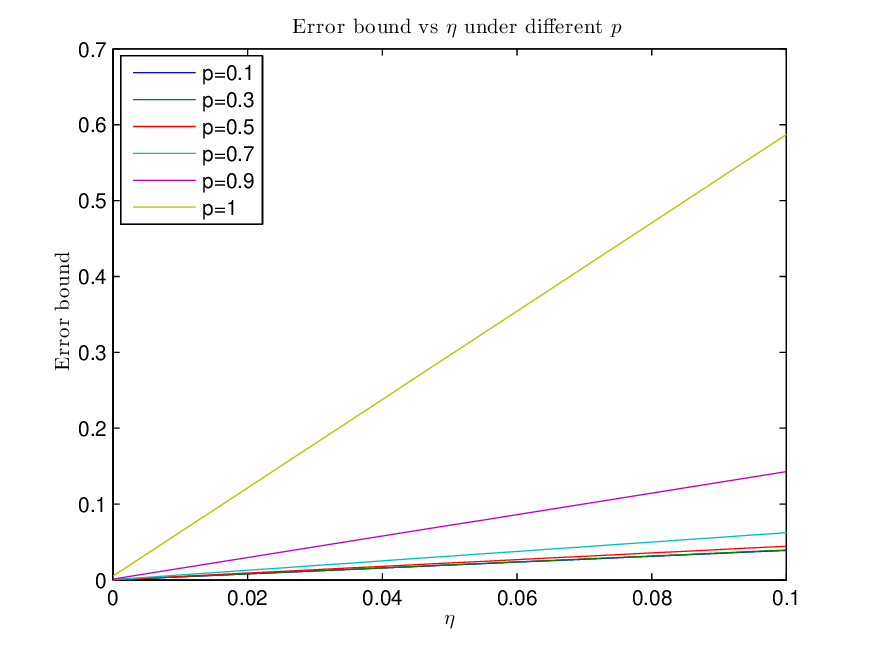}
\caption{Error curve for $0<\eta<0.01$ with $p$=0.1,0.3,0.5,0.7,0.9 and 1, respectively.}
\end{minipage}
\end{figure*}

\section{Conclusion}
Compared with sparse and cosaprse we find that both sparse and cosparse representation are often used in signal processing and machine learning applications, where they can be used to extract features or patterns from data, or to compress the data for storage or transmission. Sparse representation is particularly useful for data that has a sparse structure, where most of the coefficients in the representation are expected to be zero, while cosparse representation is particularly useful for data that has a structured, redundant, or repeating structure, where the atoms in the dictionary can be shared among multiple basis functions.

In this paper, under certain conditions on the measurement matrix $A$ and the cosparsity structure of $x$ in the analysis domain $\Omega$, the solution $\hat{x}$ to this optimization problem satisfies the error bound \eqref{E20}.
This demonstrates that the $\ell_p(0<p<1)$ cosparse optimization can stably recover the cosparse vector $x$ up to a precision level determined by the measurement noise $\sigma$ and scaling in terms of the cosparsity $s$ and number of measurements $m$.
Moreover, the $\Omega$-RIP constant $\delta_{7s}$ given in this paper has a wider range than existing literature \cite{Cand}.
When $p=0.5$ and $\delta_{7s}= 0.5$, the error constants $C_0=3.8273$ and $C_1=0.8840$ given in this paper are better than those in the existing literature \cite{Cand} and \cite{LiSong1}.
Also, experiments show that when $0<p<1$, the error result of the non-convex relaxation method in this paper is significantly smaller than that of the convex relaxation method.
Therefore, for the cosparse optimization, the $\ell_p(0<p<1)$ method provides a robust reconstruction for cosparse signal recovery.

\section*{Funding}

This work was supported by the National Natural Science Foundation of China under grant No. 12226323 and No. 12226315, and Henan Province Undergraduate College Youth Backbone Teacher Training Program.

\section*{Acknowledgments}

The authors wish to thank Professor Zheng-Hai Huang for providing his valuable comments which have significantly improved the quality of this paper.


\end{document}